\documentclass[12pt]{article}
\usepackage{amsmath,amsthm,latexsym,amsbsy,amsfonts}
\usepackage{color}
\usepackage{pstricks}
\usepackage{graphicx,epsfig}
\newcommand{\C}{{\mathbb C}}
\newcommand{\R}{{\mathbb R}}

\newcommand{\mH}{\mathbb H}
\newcommand{\mL}{\mathbb L}
\newcommand{\mT}{\mathbb T}
\newcommand{\mV}{\mathbb V}

\newcommand{\mY}{\mathbb Y}
\newcommand{\mW}{\mathbb W}

\newcommand{\inpro}[2]{\left\langle{#1},{#2}\right\rangle}

\newcommand{\norm}[2]{\|{#1}\|_{#2}}

\newcommand{\bignorm}[2]{\Big\|{#1}\Big\|_{#2}}

\newcommand{\sig}{\sigma}
\newcommand{\lam}{\lambda}

\newcommand{\vecH}{\boldsymbol{H}}

\newcommand{\veca}{\boldsymbol{a}}
\newcommand{\vecb}{\boldsymbol{b}}

\newcommand{\vecm}{\boldsymbol{m}}
\newcommand{\vecn}{\boldsymbol{n}}

\newcommand{\vecu}{\boldsymbol{u}}
\newcommand{\vecv}{\boldsymbol{v}}
\newcommand{\vecw}{\boldsymbol{w}}
\newcommand{\vecx}{\boldsymbol{x}}

\newcommand{\vectau}{\boldsymbol{\tau}}

\newcommand{\vecpsi}{\boldsymbol{\psi}}

\newcommand{\vecphi}{\boldsymbol{\phi}}
\newcommand{\veczeta}{\boldsymbol{\zeta}}

\DeclareMathOperator{\dive}{{div \/}}

\DeclareMathOperator{\curl}{{curl \/}}

\newcommand{\cE}{{\cal E}}

\newcommand{\cM}{{\cal M}}
\newcommand{\cN}{{\cal N}}

\newcommand{\goto}{\rightarrow}

\newcommand{\p}{\partial}
\newcommand{\wtd}{\widetilde}

\newcommand{\ds}{\, ds}
\newcommand{\dt}{\, dt}

\newcommand{\dvx}{\, d\vecx}

\newcommand{\nn}{\nonumber}

\numberwithin{equation}{section}
\newtheorem{theorem}{Theorem}[section]
\newtheorem{lemma}[theorem]{Lemma}

\newtheorem{remark}[theorem]{Remark}
\newtheorem{definition}[theorem]{Definition}

\title{A convergent finite element approximation for
the quasi-static Maxwell--Landau--Lifshitz--Gilbert equations
\thanks{Supported by the Australian Research Council
under grant number DP120101886}
}
\author{Kim-Ngan Le
         \thanks{School of Mathematics and Statistics,
                 The University of New South Wales,
                 Sydney 2052, Australia.
                 Email:
                 {\tt n.le-kim@student.unsw.edu.au,
                 thanh.tran@unsw.edu.au}
               }
   \and T. Tran $^\dagger$} 
\newtheorem{algorithm}{Algorithm}[section]

\begin{document}
\maketitle
\pagenumbering{arabic}
\begin{abstract}
We propose a $\theta$-linear scheme for the numerical solution 
of the quasi-static Maxwell--Landau--Lifshitz--Gilbert (MLLG) 
equations. Despite the strong nonlinearity of the 
Landau--Lifshitz--Gilbert equation, the proposed method results 
in a linear system at each time step. We prove that as the 
time and space steps tend to zero (with no further 
conditions when $\theta\in(\frac{1}{2},1]$), the finite element 
solutions converge weakly to a weak solution of the 
MLLG equations. Numerical results are presented to show 
the applicability of the method. 

{\bf Key words}: Maxwell--Landau--Lifshitz--Gilbert, finite
element, ferromagnetism

{\bf AMS suject classifications}: 65M12, 35K55
\end{abstract}
\section{ Introduction}
The Maxwell--Landau--Lifshitz--Gilbert (MLLG)
equations describe the electromagnetic behavior of a 
ferromagnetic material. In this paper, for simplicity, we suppose that there is a bounded cavity $\wtd D\subset\R^3$ (with perfectly conducting outer surface $\p\wtd D$) into which a ferromagnet $D$ is embedded. We assume further that $\wtd D\backslash \bar{D}$ is a vacuum.
We will consider the quasi-static case of the MLLG system. Letting $D_T:= (0,T)\times D$ and $\wtd D_T:= (0,T)\times \wtd D$, the magnetization field $\vecm : D_T\goto \mathbb{S}^2$, where $\mathbb{S}^2$ is the unit sphere in $\R^3$, and the magnetic field $\vecH:\wtd D_T\goto \mathbb{R}^3$ satisfy
\begin{align}
&\vecm_t
=
\lambda_1\vecm\times\vecH_{\text{eff}}
-
\lambda_2\vecm\times(\vecm\times\vecH_{\text{eff}})
\quad\text{ in } D_T,\label{E:LL}\\
&\mu_0\vecH_t
+
\sig\nabla\times(\nabla\times\vecH) 
=
-\mu_0\wtd\vecm_t
\quad\text{ in } \wtd D_T,\label{E:Maxwell2}
\end{align}
in which $\lambda_1\not=0$, $\lambda_2 >0$, $\sigma\ge 0$, and $\mu_0> 0$ are constants. Here $\wtd\vecm:\wtd D_T\goto \R^3$ is the zero extension of $\vecm$ onto $\wtd D_T$, i.e.,
\begin{equation*}
\wtd\vecm(t,\vecx)
=
\begin{cases}
\vecm(t,\vecx),&\quad (t,\vecx)\in D_T
\\
\hfill 0, &\quad (t,\vecx)\in \wtd D_T\backslash D_T.
\end{cases}
\end{equation*}
For simplicity the effective field $\vecH_{\text{eff}}$ is taken to be $\vecH_{\text{eff}}=\Delta\vecm+\vecH$.

The system \eqref{E:LL}--\eqref{E:Maxwell2} is supplemented with initial conditions
\begin{equation}\label{InCond}
\vecm(0,.)=\vecm_0\text{ in }D\quad\text{and}\quad
\vecH(0,.)=\vecH_0\text{ in }\wtd D,
\end{equation}
and boundary conditions
\begin{equation}\label{Cond}
\partial_n\vecm=0 
\text{ on } (0,T)\times\partial D\quad\text{and}\quad
(\nabla\times\vecH)\times n =0 
\text{ on } (0,T)\times\partial \wtd D.
\end{equation}

The equation~\eqref{E:LL} is the first dynamical model for the precessional motion of a magnetization, suggested by Landau and Lifshitz~\cite{LL35}.
The existence and
uniqueness of a {\em local} strong solution 
of~\eqref{E:LL}--\eqref{Cond} is shown by Cimr{\'a}k~\cite{CimExist07}. He also
proposes~\cite{CimError07} a finite element method to
approximate this local solution and provides error
estimation. 

Gilbert introduces a different approach for description of damped precession in~\cite{Gil55}:
\begin{align}\label{E:LLG}
\lam_1\vecm_t + \lam_2 \vecm\times\vecm_t
=
\mu\vecm\times\vecH_{\text{eff}}, 
\end{align}
in which $\mu=\lam_1^2+\lam_2^2$.
A proof of the equivalence between \eqref{E:LLG} and \eqref{E:LL} can be found in ~\cite{LeTran_2012}.
It is easier to numerically solve~\eqref{E:LLG} than~\eqref{E:LL} because the latter has a double cross term, namely $\vecm\times(\vecm\times\vecH_{\text{eff}})$.

Instead of solving~\eqref{E:LL}--\eqref{Cond}, Ba\v{n}as, Bartels and Prohl~\cite{BBP08}
propose an implicit nonlinear scheme to solve
problem~\eqref{E:Maxwell2}--\eqref{E:LLG}, and prove
that the finite element solution converges to a weak
{\em global} solution of the problem. Their method
requires a condition on the time step $k$ and space step
$h$ (namely $k=O(h^2)$) for the convergence of the {\em
nonlinear} system of equations resulting from the
discretization.

Following the idea  developed by Alouges
and Jaison~\cite{Alo08} for the
Landau--Lifshitz--Gilbert (LLG) equation~\eqref{E:LLG}, 
we propose a $\theta$-{\em linear} finite
element scheme to find a weak global solution 
to~\eqref{E:Maxwell2}--\eqref{E:LLG}. We prove that the numerical solutions  converge to a weak solution of the problem with no condition imposed on time step and space step as $\theta\in(\frac{1}{2},1]$. 
It is required that  $k=o(h^2)$ when
$\theta\in[0,\frac{1}{2})$, and $k=o(h)$ when 
$\theta=\frac{1}{2}$. 
The implementation aspect of the algorithm is reported 
in~\cite{LeTran_2012} where no convergence analysis is
carried out.

The paper is organized as follows. Weak solutions of the 
MLLG equations are defined in Section \ref{Sec:Def}. We
also introduce in this section the $\theta$-linear finite 
element scheme. Some technical lemmas are presented in
Section \ref{Sec:Pre}. In Section \ref{Sec:WSo}, we
prove that the finite element solutions converge to a weak 
solution of the MLLG equations. 
Numerical experiments are presented in the last section.

\section{ Weak solutions and finite element schemes}\label{Sec:Def}
Before presenting the definition of a weak solution to the MLLG equations, it is necessary to introduce some function spaces and  to assume some  conditions on the initial functions $\vecm_0$ and $\vecH_0$.

The function spaces $\mH^1(D,\R^3)$ and $\mH(\curl;\wtd D)$ are defined as follows:
\begin{align*}
\mH^1(D,\R^3)
&=
\left\{ \vecu\in\mL^2(D,\R^3) : \frac{\p\vecu}{\p
x_i}\in \mL^2(D,\R^3)\quad\text{for } i=1,2,3.
\right\},\\
\mH(\curl;\wtd D)
&=
\left\{
 \vecu\in\mL^2(\wtd D,\R^3) : \nabla\times\vecu\in \mL^2(\wtd D,\R^3)
\right\}.
\end{align*}
Here, for a domain $\Omega\subset\R^3$, $\mL^2(\Omega,\R^3)$ is the usual space of Lebesgue squared
integrable functions defined on $\Omega$ and taking values in $\R^3$.
Throughout this paper, we denote
\begin{gather*}
\inpro{\cdot}{\cdot}_{\Omega}:=\inpro{\cdot}{\cdot}_{\mL^2(\Omega,\R^3)}
\quad\text{and}\quad
\|\cdot\|_{\Omega}:=\|\cdot\|_{\mL^2(\Omega,\R^3)}.
\end{gather*}

In order to define a weak solution of MLLG equations, 
we assume that the given functions $\vecm_0$ and $\vecH_0$ satisfy
\begin{equation}\label{E:Cond1}
\vecm_0\in\mH^1(D,\R^3),\quad
|\vecm_0|=1\text{\ a.e. in }D\quad\text{and}\quad
\vecH_0\in\mH(\curl;\wtd D).
\end{equation}
For physical reasons (see~\cite{Boling08}), these initial fields must satisfy
\begin{equation}\label{E:Cond2}
\dive (\vecH_0+\chi_D\vecm_0)=0\text{ in } \wtd D
\quad\text{and}\quad
(\vecH_0+\chi_D\vecm_0)\cdot\vecn = 0\text{ on }\partial \wtd  D.
\end{equation}

Since equations \eqref{E:LL} and \eqref{E:LLG} are equivalent (a proof of which can be found in~\cite{LeTran_2012}), instead of solving~\eqref{E:LL}--\eqref{Cond} we solve~\eqref{E:Maxwell2}--\eqref{E:LLG}. A weak solution of the problem is defined in the following definition.
\begin{definition}\label{Def:WeakSo1}
Let the initial data $(\vecm_0,\vecH_0)$ satisfy~\eqref{E:Cond1} and~\eqref{E:Cond2}. Then 
$(\vecm,\vecH)$ is call a weak solution to~\eqref{E:Maxwell2}--\eqref{E:LLG} if, for all $T>0$, there hold
\begin{enumerate}
\item
$
\vecm\in\mH^1(D_T,\mathbb R^3))\text{ and } |\vecm|=1
\text{ a.e. in } D_T;
$
\item
$
\vecH,\vecH_t,\nabla\times\vecH\in\mL^2(\wtd D_T,\mathbb R^3)
$;
\item
for all $\vecphi\in\mathbb C^{\infty}(D_T)$ and $\veczeta\in\mathbb C^{\infty}(\wtd D_T)$,
\begin{align}\label{wE:LLG2}
\lambda_1\inpro{\vecm_t}{\vecphi}_{D_T}
+
\lambda_2\inpro{\vecm\times\vecm_t}{\vecphi}_{D_T}
=
\mu\inpro{\nabla\vecm}{\nabla(\vecm\times\vecphi)}_{D_T}
+
\mu\inpro{\vecm\times\vecH}{\vecphi}_{D_T}
\end{align}
and
\begin{equation}\label{wE:Maxwell2}
\mu_0\inpro{\vecH_t}{\veczeta}_{\wtd D_T}
+\sig\inpro{\nabla\times\vecH}{\nabla\times\veczeta}_{\wtd D_T}
=
-\mu_0\inpro{\wtd\vecm_t}{\veczeta}_{\wtd D_T},
\end{equation}
where $\mu=\lambda_1^2+\lambda_2^2$;
\item
in the sense of traces there holds
\begin{equation}\label{equ:m0}
\vecm(0,\cdot)=\vecm_0, 
\end{equation}
\item for almost all $T'\in(0,T)$,
\begin{align}\label{InE:Energy}
\cE(T')
+\lambda_2\mu^{-1}\|\vecm_t\|^2_{D_{T'}}
+\lambda_2\mu^{-1}\|\vecH_t\|^2_{\wtd D_{T'}}
+2\mu_0^{-1}\sigma\|\nabla\times\vecH\|^2_{\wtd D_{T'}}
\leq
\cE(0),
\end{align}
where
\[
\cE(T')=
\|\nabla\vecm(T')\|^2_{D}
+\|\vecH(T')\|^2_{\wtd D}
+\lambda_2\mu^{-1}\mu_0^{-1}\sigma\|\nabla\times\vecH(T')\|^2_{\wtd D}.
\]
\end{enumerate}
\end{definition}

We next introduce the $\theta$-linear finite element scheme 
which approximates a weak solution $(\vecm,\vecH)$ 
defined in Definition~\ref{Def:WeakSo1}.

Let $\mT_h$ be a regular tetrahedrization of the domain $\wtd D$ into tetrahedra of maximal mesh-size $h$, 
and let $\mT_h|_D$ be  its restriction to $D\subset\wtd D$. We denote by $\cN_h := \{\vecx_1,\ldots,\vecx_N\}$ the set of vertices and by $\cM_h :=\{ e_1, \ldots , e_M \}$ the set of edges.

To discretize the LLG equation~\eqref{wE:LLG2}, we introduce the finite element space $\mV_h\subset\mH^1(D,\R^3)$ which is the space of all continuous piecewise linear functions on $\mT_h|_D$. A basis for $\mV_h$ can be chosen to be $(\phi_n)_{1\leq n\leq N}$, where
 $\phi_n(\vecx_m)=\delta_{n,m}.$  Here $\delta_{n,m}$ stands for
the Kronecker symbol.
The interpolation operator from
$\C^0(D,\R^3)$ onto  $\mV_h$ is denoted by $I_{\mV_h}$,
\[
I_{\mV_h}(\vecv)=\sum_{n=1}^N \vecv(\vecx_n)\phi_n(\vecx)
\quad\forall \vecv\in \mathbb C^0(D,\mathbb R^3) .
\]

To discretize Maxwell's equation~\eqref{wE:Maxwell2}, we
use the space $\mY_h$ of lowest order edge elements of 
Nedelec's first family~\cite{Monk03}. 
It is known~\cite{Monk03} that $\mY_h$ is a subspace of 
$\mH(\curl;\wtd D)$ and that the set
$\{\vecpsi_1,\ldots,\vecpsi_M\}$ is a basis for $\mY_h$
if it satisfies
\begin{equation}\label{equ:psiq}
\begin{aligned}
&\vecpsi_q
\in
\{\vecpsi : \wtd D\goto\R^3 \ | \
\vecpsi|_{K}(\vecx) = \veca_K+\vecb_K\times\vecx,
\ \veca_K,\vecb_K\in\R^3, \forall K\in\mT_h\}, \\
&\int_{e_p}\vecpsi_q\cdot\vectau_p\ds 
= 
\delta_{qp},
\end{aligned}
\end{equation}
where $ \vectau_p $ is the unit vector in the direction of edge $e_p$. 
We also define the following interpolation operator $I_{\mY_h}$ from $\C^{\infty}(\wtd D)$ onto $\mY_h$,
\[
I_{\mY_h}(\vecu)=\sum_{q=1}^M u_q \vecpsi_q
\quad\forall \vecu\in \mathbb C^{\infty}(\wtd D,\mathbb R^3),
\]
where
$
u_q=\int_{e_q}\vecu\cdot\vectau_q\ds.
$

Fixing a positive integer $J$, we choose the time step
$k$ to be $k=T/J$ and define $t_j=jk$, $j=0,\cdots,J$.  For $j=1,2,\ldots,J$,
the functions $\vecm (t_j,\cdot)$ and $\vecH(t_j,\cdot)$ are approximated by $\vecm^{(j)}_h\in\mV_h$ and $\vecH_h^{(j)}\in\mY_h$, respectively. 

We define the space $\mW^{(j)}_h$ by
\begin{equation*}
 \mW_h^{(j)}
 :=
 \left\{\vecw\in \mV_h \mid 
 \vecw(\vecx_n)\cdot\vecm_h^{(j)}(\vecx_n)=0,
 \ n = 1,\ldots,N \right\},
\end{equation*}
and denote 
\[
\vecH^{(j+1/2)}_h:=\frac{\vecH^{(j+1)}_h+\vecH^{(j)}_h}{2}
\quad\text{and}\quad
d_t\vecH^{(j+1)}_h
:=k^{-1}(\vecH^{(j+1)}_h-\vecH^{(j)}_h).
\]
\begin{algorithm}\label{Algo:1}
\mbox{}
\begin{description}
\item[Step 1:]
Set $j=0$.
Choose $\vecm^{(0)}_h=I_{\mV_h}\vecm_0$ and
$\vecH^{(0)}_h=I_{\mY_h}\vecH_0$.
\item[Step 2:] \label{A:2}
Find $(\vecv_h^{(j+1)},\vecH^{(j+1)}_h)\in
\mW_h^{(j)}\times\mY_h$ satisfying
\begin{align}\label{disE:LLG}
 &\lambda_2
 \inpro{\vecv_h^{(j+1)}}{\vecw_h^{(j)}}_{D}
 -\lambda_1
 \inpro{\vecm_h^{(j)}\times\vecv_h^{(j+1)}}{\vecw_h^{(j)}}_{D}\nn\\
 &=
 -\mu 
 \inpro{\nabla (\vecm_h^{(j)}+k\theta \vecv_h^{(j+1)})}
 {\nabla\vecw_h^{(j)}}_{D} 
 +
 \mu
 \inpro{\vecH^{(j+1/2)}_h}{\vecw_h^{(j)}}_{D}
 \quad\forall \vecw_h^{(j)}\in\mW_h^{(j)},
 \end{align}
and
\begin{align}\label{disE:NewMax2}
 \mu_0\inpro{d_t\vecH^{(j+1)}_h}{\veczeta_h}_{\wtd D}
 &+\sigma\inpro{\nabla\times\vecH^{(j+1/2)}_h}
 {\nabla\times\veczeta_h}_{\wtd D}\nn\\
 &=
 -\mu_0\inpro{\vecv^{(j+1)}_h}{\veczeta_h}_{\wtd D}
 \quad\forall \veczeta_h\in\mY_h.
\end{align}
\item[Step 3:] \label{A:4}
Define
\begin{equation*}
\vecm_h^{(j+1)}(\vecx)
:=
\sum_{n=1}^N
\frac{\vecm_h^{(j)}(\vecx_n)+k\vecv_h^{(j+1)}(\vecx_n)}
{\left|\vecm_h^{(j)}(\vecx_n)+k\vecv_h^{(j+1)}(\vecx_n)\right|}
\phi_n(\vecx).
\end{equation*}
\item[Step 4:] 
Set $j=j+1$, and return to Step~\ref{A:2} if $j<J$. Stop if $j=J$.
\end{description}
\end{algorithm}
\noindent
The parameter $\theta$ in~\eqref{disE:LLG} can be chosen arbitrarily in
$[0,1]$. The method is explicit when $\theta=0$ and 
fully implicit when $\theta=1$.

By the Lax--Milgram Theorem, for each $j>0$ 
there exists a unique solution 
$(\vecv_h^{(j+1)},\vecH^{(j+1)}_h)\in \mW_h^{(j)}\times\mY_h$ 
of equations~\eqref{disE:LLG}--\eqref{disE:NewMax2}.
Since $\left|\vecm_h^{(j)}(\vecx_n)\right|=1$ and 
$\vecv_h^{(j+1)}(\vecx_n)\cdot\vecm_h^{(j)}(\vecx_n)=0$ for all
$n=1,\ldots,N$ , there holds
\begin{equation}\label{InE:CondGr1}
\left
|\vecm_h^{(j)}(\vecx_n)+k\vecv_h^{(j+1)}(\vecx_n)\right|
\ge 1.
\end{equation}
Therefore, the algorithm is well defined. There also
holds $\left|\vecm_h^{(j+1)}(\vecx_n)\right|=1$ for $n=1,\cdots,N.$

\section{Some  technical lemmas }\label{Sec:Pre}
In this section we present some lemmas which will be used in the rest of the paper. We start by recalling the following lemma proved in~\cite{Bart05}.
\begin{lemma}\label{lem:4.0}
If there holds
\begin{equation}\label{E:CondTe}
\int_D \nabla\phi_i\cdot\nabla\phi_j\dvx \leq 0
\quad\text{for all}\quad i,j \in \{1,2,\cdots,J\}\text{ and } i\not= j ,
\end{equation}
then for all $\vecu\in\mV_h$ satisfying $|\vecu(\vecx_l)|\geq 1$, $ l=1,2,\cdots,J$, there holds
\begin{equation}\label{E:InE}
\int_D\left|\nabla I_h\left(\frac{\vecu}{|\vecu|}\right)\right|^2\dvx
\leq
\int_D|\nabla\vecu|^2\dvx.
\end{equation}
\end{lemma}
\noindent
Condition~\eqref{E:CondTe} holds if all dihedral angles of 
the tetrahedra in $\mT_h|_D$ are less than or equal 
to $\pi/2$; see~\cite{Bart05}. 
In the sequel we assume that~\eqref{E:CondTe} holds.

The next lemma defines a discrete $\mL^p$-norm in
$\mV_h$ which is equivalent to the usual $\mL^p$-norm.
\begin{lemma}\label{lem:3.4}
There exist $h$-independent positive constants $C_1$ and $C_2$ such that for all $p\in[1,\infty]$ and $\vecu\in\mV_h$ there holds
\begin{equation*}
C_1\|\vecu\|^p_{\mL^p(\Omega)}
\leq
h^d
\sum_{n=1}^N |\vecu(\vecx_n)|^p
\leq
C_2\|\vecu\|^p_{\mL^p(\Omega)},
\end{equation*}
where $\Omega\subset\R^d$, d=1,2,3.
\end{lemma}
A proof of this lemma for $p=2$ and $d=2$ can be found in~\cite[Lemma 7.3]{Johnson87} or~\cite[Lemma 1.12]{Chen05}. The result for general values of $p$ and $d$ can be obtained in the same manner.

The following lemma can be proved by using the technique in~\cite[Lemma 7.3]{Johnson87} .
\begin{lemma}\label{lem:3.5}
There exists an $h$-independent positive constant $C$ such that for each tetrahedron $K\in \mT_h$ and $\vecv\in\mV_h$ there holds
\begin{equation*}
\left| |\vecv(\vecx)| - |\vecv(\vecx_i)|\right|
\leq
Ch|\nabla\vecv(\vecx)|
\quad\text{for all } \vecx\in K,
\end{equation*}
where $\{\vecx_i\}_{i=1,2,3}$ are the vertices of $K$.
\end{lemma}

Finally the following lemma is elementary; the proof of
which is included for completeness.
\begin{lemma}\label{lem:3.3}
The solutions 
$\left(\vecm_h^{(j)},\vecv_h^{(j+1)}\right)$,
$j=0,1,\cdots,J$, 
obtained from Algorithm~\ref{Algo:1} satisfy
\begin{equation}\label{InE:lem1App}
\left|
\frac{\vecm_h^{(j+1)}(\vecx_n)-\vecm_h^{(j)}(\vecx_n)}{k}
\right|
\leq
\left|
\vecv_h^{(j+1)}(\vecx_n)
\right|
\quad
\forall n=1,2,\cdots, N, \quad j=0,\ldots,J.
\end{equation}
\end{lemma}
\begin{proof}
By using the definition of $\vecm_h^{(j+1)}$, the
property
$\vecm_h^{(j)}(\vecx_n) \cdot \vecv_h^{(j+1)}(\vecx_n) = 0$,
and the identity 
\[
|\vecm_h^{(j)}(\vecx_n)
+
k\vecv_h^{(j+1)}(\vecx_n)|
=
\sqrt{1+k^2|\vecv_h^{(j+1)}(\vecx_n)|^2}
\]
we obtain
\begin{align*}
\left|
\frac{\vecm_h^{(j+1)}(\vecx_n)-\vecm_h^{(j)}(\vecx_n)}{k}
\right|^2
&=
\left|
\frac{\vecm_h^{(j)}(\vecx_n)+k\vecv_h^{(j+1)}(\vecx_n)}
{k\left |\vecm_h^{(j)}(\vecx_n)+k\vecv_h^{(j+1)}(\vecx_n)\right|}
-
\frac{\vecm_h^{(j)}(\vecx_n)}{k}
\right|^2\\
&=
\frac{\left|\vecm_h^{(j)}(\vecx_n)
\big(1-|\vecm_h^{(j)}(\vecx_n)+k\vecv_h^{(j+1)}(\vecx_n)|\big)
+
k\vecv_h^{(j+1)}(\vecx_n)\right|^2}
{k^2|\vecm_h^{(j)}(\vecx_n)+k\vecv_h^{(j+1)}(\vecx_n)|^2}\\
&=
\frac{2+2k^2\left|\vecv_h^{(j+1)}(\vecx_n)\right|^2
-
2\sqrt{1+k^2|\vecv_h^{(j+1)}(\vecx_n)|^2}}
{k^2\left(1+k^2|\vecv_h^{(j+1)}(\vecx_n)|^2\right)}\\
&=
2\frac{\sqrt{1+k^2|\vecv_h^{(j+1)}(\vecx_n)|^2}-1}
{k^2\sqrt{1+k^2|\vecv_h^{(j+1)}(\vecx_n)|^2}}.
\end{align*}
Using the fact that 
\[
2\leq\sqrt{1+k^2|\vecv_h^{(j+1)}(\vecx_n)|^2}+1
\]
we deduce
\begin{align*}
\left|
\frac{\vecm_h^{(j+1)}(\vecx_n)-\vecm_h^{(j)}(\vecx_n)}{k}
\right|^2
&\leq
\frac{\left(\sqrt{1+k^2|\vecv_h^{(j+1)}(\vecx_n)|^2}+1\right)
\left(\sqrt{1+k^2|\vecv_h^{(j+1)}(\vecx_n)|^2}-1\right)}
{k^2}\nn\\
&=
\left|
\vecv_h^{(j+1)}(\vecx_n) 
\right|^2,
\end{align*}
proving the lemma.
\end{proof}

In the following section, we show that our numerical solution converges to a weak solution  of the problem~\eqref{E:Maxwell2}--\eqref{E:LLG}.
\section{Existence of weak solutions}\label{Sec:WSo}
The next lemma provides a bound in the $\mL^2$-norm for 
the discrete solutions.
\begin{lemma}\label{lem:4.1}
The sequence 
$\left\{
\big(\vecm_h^{(j)},\vecv_h^{(j+1)},\vecH^{(j)}_h\big)
\right\}_{j=0,1,\cdots,J}$ produced by Algorithm \ref{Algo:1} satisfies
\begin{align}\label{E:4.1}
\cE_h^{(j)}
+
C\sum_{i=0}^{j-1} k\|\vecv_h^{(i+1)}\|^2_{D}
&+
\lambda_2\mu^{-1}
\sum_{i=0}^{j-1} k\|d_t\vecH^{(i+1)}_h\|^2_{\wtd D}\nn\\
&+
2\mu_0^{-1}\sigma
\sum_{i=0}^{j-1} k\|\nabla\times\vecH^{(i+1/2)}_h\|^2_{\wtd D}
\leq
\cE_h^0,
\end{align}
where
\[
\cE_h^{(j)}=
\|\nabla\vecm^{(j)}_h\|^2_{D}
+\|\vecH^{(j)}_h\|^2_{\wtd D}
+\lambda_2\mu^{-1}\mu_0^{-1}\sigma\|\nabla\times\vecH^{(j)}_h\|^2_{\wtd D},
\]
and
\[
C=
\begin{cases}
\lambda_2\mu^{-1}, &\quad\theta\in[\frac{1}{2},1] \\
\lambda_2\mu^{-1}-(1-2\theta)C_1kh^{-2}, &\quad\theta\in[0,\frac{1}{2}),
\end{cases}
\]
in which $C_1$ is a positive constant which is independent with $j$, $k$ and $h$.
\end{lemma}
\begin{proof}
Choosing $\vecw^{(j)}_h=\vecv^{(j+1)}_h$ in~\eqref{disE:LLG} and $\veczeta_h=\vecH^{(j+1/2)}_h$ in~\eqref{disE:NewMax2}, we obtain
\begin{align}
&\lambda_2
\|\vecv_h^{(j+1)}\|^2_{D}
+k\theta\mu 
\|\nabla\vecv_h^{(j+1)}\|^2_{D}
=
-\mu\inpro{\nabla\vecm_h^{(j)}}{\nabla\vecv_h^{(j+1)}}_{D} 
+
\mu
\inpro{\vecH^{(j+1/2)}_h}{\vecv_h^{(j+1)}}_{D}\label{E:4.2}\\
&\frac{\mu_0}{2} d_t\|\vecH^{(j+1)}_h\|^2_{\wtd D}
+\sig\|\nabla\times\vecH^{(j+1/2)}_h\|^2_{\wtd D}
=
-\mu_0\inpro{\vecv^{(j+1)}_h}{\vecH^{(j+1/2)}_h}_{\wtd D}.\label{E:4.3}
\end{align}
Multiplying $\mu\mu_0^{-1}$ to both sides of \eqref{E:4.3} and adding the resulting equation to \eqref{E:4.2}, we deduce
\begin{align}\label{E:4.4}
\lambda_2
\|\vecv_h^{(j+1)}\|^2_{D}
&+k\theta\mu 
\|\nabla\vecv_h^{(j+1)}\|^2_{D}
+\frac{\mu}{2} d_t\|\vecH^{(j+1)}_h\|^2_{\wtd D}\nn\\
&+\mu\mu_0^{-1}\sig\|\nabla\times\vecH^{(j+1/2)}_h\|^2_{\wtd D}
=
-\mu\inpro{\nabla\vecm_h^{(j)}}{\nabla\vecv_h^{(j+1)}}_{D}.  
\end{align}
Since $\vecm^{(j)}_h+k\vecv^{(j+1)}_h\in\mV_h$ and 
\[
\vecm_h^{(j+1)}
=
I_h\left(\frac{\vecm^{(j)}_h+k\vecv^{(j+1)}_h}{|\vecm^{(j)}_h+k\vecv^{(j+1)}_h|}\right),
\]
it follows from~\eqref{InE:CondGr1} and Lemma~\ref{lem:4.0} that
\begin{equation*}
\|\nabla\vecm^{(j+1)}_h\|^2_{D}
\leq
\|\nabla(\vecm^{(j)}_h+k\vecv^{(j+1)}_h)\|^2_{D}.
\end{equation*}
Equivalently, we have
\begin{equation}\label{InE2}
\|\nabla\vecm^{(j+1)}_h\|^2_{D}
\leq
\|\nabla\vecm^{(j)}_h\|^2_{D}
+k^2\|\nabla\vecv^{(j+1)}_h\|^2_{D}
+2k\inpro{\nabla\vecm^{(j)}_h}{\nabla\vecv^{(j+1)}_h}_{D}.
\end{equation}
Equality \eqref{E:4.4} is used to obtain from~\eqref{InE2} the following inequality
\begin{align*}
\|\nabla\vecm^{(j+1)}_h\|^2_{D}
&\leq
\|\nabla\vecm^{(j)}_h\|^2_{D}
-k^2(2\theta-1)\|\nabla\vecv^{(j+1)}_h\|^2_{D}
-2k\lambda_2\mu^{-1}\|\vecv_h^{(j+1)}\|^2_{D}\\
&\quad
-kd_t\|\vecH^{(j+1)}_h\|^2_{\wtd D}
-2k\mu_0^{-1}\sig\|\nabla\times\vecH^{(j+1/2)}_h\|^2_{\wtd D}.
\end{align*}
Hence,
\begin{align}\label{E:4.5}
\|\nabla\vecm^{(j+1)}_h\|^2_{D}
&+\|\vecH^{(j+1)}_h\|^2_{\wtd D}
+2k\lambda_2\mu^{-1}\|\vecv_h^{(j+1)}\|^2_{D}
+2k\mu_0^{-1}\sig\|\nabla\times\vecH^{(j+1/2)}_h\|^2_{\wtd D}\nn\\
&+k^2(2\theta-1)\|\nabla\vecv^{(j+1)}_h\|^2_{D}
\leq
\|\nabla\vecm^{(j)}_h\|^2_{D}
+\|\vecH^{(j)}_h\|^2_{\wtd D}.
\end{align}
Next choosing $\veczeta_h=d_t\vecH^{(j+1)}_h$ in equation \eqref{disE:NewMax2}, we obtain
\begin{align*}
2k\mu_0\|d_t\vecH^{(j+1)}_h\|^2_{\wtd D}
+\sig\|\nabla\times\vecH^{(j+1)}_h\|^2_{\wtd D}
=
&\sig\|\nabla\times\vecH^{(j)}_h\|^2_{\wtd D}\\
&-2k\mu_0\inpro{\vecv^{(j+1)}_h}{d_t\vecH^{(j+1)}_h}_{\wtd D}.
\end{align*}
The term $-2k\mu_0\inpro{\vecv^{(j+1)}_h}{d_t\vecH^{(j+1)}_h}_{\wtd D}$ can be estimated by
\[
-2k\mu_0\inpro{\vecv^{(j+1)}_h}{d_t\vecH^{(j+1)}_h}_{\wtd D}
\leq
k\mu_0\|\vecv^{(j+1)}_h\|^2_{D}
+k\mu_0\|d_t\vecH^{(j+1)}_h\|^2_{\wtd D}.
\]
Therefore, we deduce
\begin{align}\label{E:4.6}
k\mu_0\|d_t\vecH^{(j+1)}_h\|^2_{\wtd D}
+\sig\|\nabla\times\vecH^{(j+1)}_h\|^2_{\wtd D}
\leq
\sig\|\nabla\times\vecH^{(j)}_h\|^2_{\wtd D}
+k\mu_0\|\vecv^{(j+1)}_h\|^2_{D}.
\end{align}
Multiplying $\lambda_2\mu^{-1}\mu_0^{-1}$ to both sides of 
\eqref{E:4.6} and adding the resulting equation to 
\eqref{E:4.5}, we obtain
\begin{align*}
&\|\nabla\vecm^{(j+1)}_h\|^2_{D}
+\|\vecH^{(j+1)}_h\|^2_{\wtd D}
+\lambda_2\mu^{-1}\mu_0^{-1}\sig\|\nabla\times\vecH^{(j+1)}_h\|^2_{\wtd D}\\
&+k\lambda_2\mu^{-1}\|\vecv_h^{(j+1)}\|^2_{D}
+k\lambda_2\mu^{-1}\|d_t\vecH^{(j+1)}_h\|^2_{\wtd D}
+2k\mu_0^{-1}\sig\|\nabla\times\vecH^{(j+1/2)}_h\|^2_{\wtd D}\\
&+k^2(2\theta-1)\|\nabla\vecv^{(j+1)}_h\|^2_{D} \\
\leq \
&\|\nabla\vecm^{(j)}_h\|^2_{D}
+\|\vecH^{(j)}_h\|^2_{\wtd D}
+\lambda_2\mu^{-1}\mu_0^{-1}\sig
\|\nabla\times\vecH^{(j)}_h\|^2_{\wtd D}.
\end{align*}
Replacing $j$ by $i$ in the above inequality and summing over $i$ from $0$ to $j-1$ yield
\begin{align}\label{E:4.7}
&\|\nabla\vecm^{(j)}_h\|^2_{D}
+\|\vecH^{(j)}_h\|^2_{\wtd D}
+\lambda_2\mu^{-1}\mu_0^{-1}\sig\|\nabla\times\vecH^{(j)}_h\|^2_{\wtd D}
+\lambda_2\mu^{-1}\sum_{i=0}^{j-1} k\|\vecv_h^{(i+1)}\|^2_{D}\nn\\
&+\lambda_2\mu^{-1}\sum_{i=0}^{j-1} k\|d_t\vecH^{(i+1)}_h\|^2_{\wtd D}
+2\mu_0^{-1}\sig\sum_{i=0}^{j-1} 
k\|\nabla\times\vecH^{(i+1/2)}_h\|^2_{\wtd D}\nn\\
&+k^2(2\theta-1)\sum_{i=0}^{j-1} \|\nabla\vecv^{(i+1)}_h\|^2_{D}
\nn\\
\le \
& \|\nabla\vecm^0_h\|^2_{D}
+\|\vecH^0_h\|^2_{\wtd D}
+\lambda_2\mu^{-1}\mu_0^{-1}\sig
\|\nabla\times\vecH^0_h\|^2_{\wtd D}.
\end{align}

When $\theta\in [\frac{1}{2},1]$, the term $k^2(2\theta-1)\sum_{i=0}^{j-1} \|\nabla\vecv^{(i+1)}_h\|^2_{D}$ is nonnegative. Hence, from inequality~\eqref{E:4.7} we obtain~\eqref{E:4.1} where $C=\lam_2\mu^{-1}$.
When $\theta\in [0,\frac{1}{2})$, using the inverse estimate we obtain
\begin{equation}\label{InE:Inverse}
C_1k^2h^{-2}(2\theta-1)\sum_{i=0}^{j-1} \|\vecv^{(i+1)}_h\|^2_{D}
\leq
k^2(2\theta-1)\sum_{i=0}^{j-1} \|\nabla\vecv^{(i+1)}_h\|^2_{D},
\end{equation}
where, $C_1$ is a positive constant which is independent with $j$, $k$ and $h$.
Hence, from inequality~\eqref{E:4.7} we obtain~\eqref{E:4.1} where $C=\lam_2\mu^{-1}-C_1kh^{-2}(1-2\theta)$.
This completes the proof of the lemma.
\end{proof}
\begin{remark}\label{Rem:1}
The constant $C$ in the above lemma is positive when
$\theta\in[1/2,1]$. When $\theta\in[0,1/2)$ the
additional condition $k=o(h^2)$ assures us that $C$ is
positive when $h$ and $k$ are sufficiently small. 
This condition will be required later in the
following lemma and theorem.
\end{remark}

The discrete solutions $\vecm_h^{(j)}$, $\vecv_h^{(j+1)}$ and $\vecH_h^{(j)}$ constructed via Algorithm~\ref{Algo:1} are interpolated in time in the following definition.
\begin{definition}\label{Def:m}
For each $t\in[0,T]$, let
$j\in \{ 0,...,J \}$ be
such that  $t \in [t_j, t_{j+1})$. We define for $t\in[0,T]$ and $\vecx\in D$
\begin{align*}
\vecm_{h,k}(t,\vecx)
&:= 
\frac{t-t_j}{k}\vecm_h^{(j+1)}(\vecx)
+
\frac{t_{j+1}-t}{k}\vecm_h^{(j)}(\vecx), \\
\vecm_{h,k}^{-}(t,\vecx) 
&:= 
\vecm_h^{(j)}(\vecx), \\
\vecv_{h,k}(t,\vecx) 
&:= 
\vecv_h^{(j+1)}(\vecx),\\
\vecH_{h,k}(t,x) 
&:= 
\frac{t-t_j}{k}\vecH_h^{(j+1)}(\vecx)
+
\frac{t_{j+1}-t}{k}\vecH_h^{(j)}(\vecx), \\
\wtd\vecH_{h,k}(t,x) 
&:= 
\frac{1}{2}\left(\vecH_h^{(j+1)}(\vecx)+\vecH_h^{(j)}(\vecx)\right), \\
\text{ and}\quad
\vecH_{h,k}^{-}(t,\vecx) 
&:= 
\vecH_h^{(j)}(\vecx).
\end{align*}
\end{definition}

The following lemma shows  that $\{\vecm_{h,k}\}$, $\{\vecm_{h,k}^{-}\}$ and $\{\vecv_{h,k}\}$ converge (up to the extraction of subsequences) as $h$ and $k$ tend  to $0$.
\begin{lemma}\label{lem:4.2}
Assume that $h$ and $k$ go to 0 with a further condition 
$k=o(h^2)$ when $\theta\in[0,\frac{1}{2})$ and no
condition otherwise.
There exist $\vecm\in\mH^1(D_T,\R^3)$ and
$\vecH\in\mH^1(0,T,\mL^2(\wtd D))$ such that  
$\nabla\times\vecH$ belongs to $\mL^2(\wtd D_T)$ and
\begin{gather}
\vecm_{h,k} \goto \vecm \text{ strongly in } \mL^2(D_T),\label{star:1} \\
\frac{\partial\vecm_{h,k}}{\partial t}  
\rightharpoonup \vecm_t
\text{ weakly in } \mL^2(D_T),\label{star:2} \\
\vecv_{h,k}
\rightharpoonup \vecm_t
\text{ weakly in } \mL^2(D_T),\label{star:3} \\
\vecm_{h,k}^- \goto \vecm \text{ strongly in } \mL^2(D_T),\label{star:4}\\
|\vecm|=1 \text{ a.e. in } D_T,\label{star:5}\\
\vecH_{h,k} \rightharpoonup \vecH \text{ weakly in } \mH^1(0,T,\mL^2(\wtd D)), \label{star:6}\\
\nabla\times\wtd\vecH_{h,k} 
\rightharpoonup \nabla\times\vecH
\text{ weakly in } \mL^2(\wtd D_T),\label{star:7} \\
\text{and}\quad\nabla\times\vecH_{h,k}^- 
\rightharpoonup \nabla\times\vecH
\text{ weakly in } \mL^2(\wtd D_T).\label{star:8}
\end{gather}
\end{lemma}
\begin{proof}
\mbox{}

\noindent
\underline{Proof of~\eqref{star:1} and~\eqref{star:2}:}

Our goal is to prove  that $\{\vecm_{h,k}\}$ is bounded 
in $\mH^1(D_T,\R^3)$ and then use the Banach--Alaoglu
Theorem. We note from
Definition~\ref{Def:m} that it suffices to prove that
\begin{align}
\norm{\vecm_h^{(j)}}{D}
&\le c, \label{equ:b1} \\
\norm{\nabla\vecm_h^{(j)}}{D}
&\le c, \label{equ:b2} \\
k \sum_{j=0}^{J-1}
\bignorm{\frac{\vecm_h^{(j+1)}-\vecm_h^{(j)}}{k}}{D}^2
&\le c,  \label{equ:b3}
\end{align}
where the generic constant $c$ is independent of $j$, $h$,
and $k$.
Indeed, it follows from Definition~\ref{Def:m} and the
Cauchy--Schwarz inequality that
\begin{align*}
\left\|\vecm_{h,k}\right\|_{D_T}^2
&\leq
ck \sum_{j=0}^{J-1}
\left(
\left\|\vecm_h^{(j+1)}\right\|_{D}^2
+
\left\|\vecm_h^{(j)}\right\|_{D}^2
\right), \\
\left\|\nabla\vecm_{h,k}\right\|_{D_T}^2
&\leq
ck \sum_{j=0}^{J-1}
\left(
\left\|\nabla\vecm_h^{(j+1)}\right\|_{D}^2
+
\left\|\nabla\vecm_h^{(j)}\right\|_{D}^2
\right), \\
\left\|
\frac{\partial\vecm_{h,k}}{\partial t}
\right\|^2_{D_T}
&=
\sum_{j=0}^{J-1}k\left\|
\frac{\vecm_h^{(j+1)}-\vecm_h^{(j)}}{k}
\right\|_{D}^2.
\end{align*}

In order to prove~\eqref{equ:b1} we note that
for every $\vecx\in D$ there are at most $4$ basis functions
$\phi_{n_1}$, $\phi_{n_2}$, $\phi_{n_3}$   and 
$\phi_{n_4}$ being nonzero at $\vecx$.
This together with $|\vecm_h^{(j)}(\vecx_{n_i})|=1$
and $\sum_{i=1}^4 \phi_{n_i}(\vecx)=1$
yields
\begin{align}\label{bound:minf}
|\vecm_h^{(j)}(\vecx)|^2
=
\left|\sum_{i=1}^{4} 
\vecm_h^{(j)}(\vecx_{n_i})\phi_{n_i}(\vecx)\right|^2
\le 1.
\end{align}
This implies \eqref{equ:b1} with a constant
$c=|D|^{1/2}$
where $|D|$ is the measure of the domain~$D$.

Inequality~\eqref{equ:b2} is proved
in~Lemma~\ref{lem:4.1}. In order to prove 
inequality~\eqref{equ:b3}, we note that
Lemma~\ref{lem:3.3} and Lemma~\ref{lem:3.4} imply
\[
\left\|
\frac{\vecm_h^{(j+1)}-\vecm_h^{(j)}}{k}
\right\|_{D}
\leq
c\left\|
\vecv_h^{(j+1)}
\right\|_{D}.
\]
By using this inequality, Lemma~\ref{lem:4.1} and
Remark~\ref{Rem:1} we deduce
\[
k \sum_{j=0}^{J-1}
\bignorm{\frac{\vecm_h^{(j+1)}-\vecm_h^{(j)}}{k}}{D}^2
\leq
k \sum_{j=0}^{J-1}
c\left\|
\vecv_h^{(j+1)}
\right\|_{D}^2
\leq
c.
\]

The Banach--Alaoglu Theorem implies the existence of
a subsequence of $\{\vecm_{h,k}\}$ which converges weakly to 
a function $\vecm\in\mH ^1(D_T)$ as $k$ and $h$ tend to
zero. This implies~\eqref{star:1} and~\eqref{star:2}.

\bigskip
\noindent
\underline{Proof of~\eqref{star:3}:}

From~\eqref{E:4.1} and and Remark~\ref{Rem:1}, it 
is straightforward to show that $\{\vecv_{h,k}\}$ is 
bounded in $\mL^2(D_T)$. Hence, there exists a subsequence 
of $\{\vecv_{h,k}\}$ which converges weakly to a function 
$\vecv\in\mL ^2(D_T)$. 
The problem reduces to proving that $\vecm_t$  equals 
$\vecv$ in $\mL^2(D_T)$. In order to show this we
choose for each $\vecpsi\in \mL^2(D_T)$ a sequence
$\{\vecpsi_i\}\in\C_0^\infty(D_T)$ converging to $\vecpsi$ 
in $\mL^2(D_T)$ as $i$ tends to infinity.
We then have
\begin{align}
|\inpro{\vecm_t-\vecv}{\vecpsi}_{D_T}|
&\leq
|\inpro{\vecm_t-\vecv}
{\vecpsi_i-\vecpsi}_{D_T}|
+
\left|\inpro{\vecm_t-\frac{\partial\vecm_{h,k}}{\partial t}}
{\vecpsi_i}_{D_T}\right| \nn\\
&\quad+
\left|\inpro{\frac{\partial\vecm_{h,k}}{\partial t}-\vecv_{h,k}}
{\vecpsi_i}_{D_T}\right|
+
|\inpro{\vecv_{h,k}-\vecv}
{\vecpsi_i}_{D_T} \nn\\
&\leq
\left\|\vecm_t-\vecv\right\|_{D_T}
\left\|\vecpsi_i-\vecpsi\right\|_{D_T}
+
\left|\inpro{\vecm_t-\frac{\partial\vecm_{h,k}}{\partial t}}
{\vecpsi_i}_{D_T}\right| \nn\\
&\quad+
\left\|\frac{\partial\vecm_{h,k}}{\partial t}-\vecv_{h,k}\right\|_{\mL^1(D_T)}
\left\|\vecpsi_i\right\|_{\mL^\infty(D_T)}
+
|\inpro{\vecv_{h,k}-\vecv}
{\vecpsi_i}_{D_T}| \nn\\
&=:
T_1 + \cdots + T_4. \label{equ:mt}
\end{align}
By letting $h,k\goto 0$ and then $i\goto\infty$ we have
$T_i\goto0$ for $i=1,2$ and $4$. It remains to show that
$T_3\goto0$.

It is clear from the definition of $\vecm_h^{(j+1)}$ in 
Algorithm~\ref{Algo:1} that 
\begin{equation}\label{E:4.8}
\left|
\vecm_h^{(j+1)}(\vecx_n)
-
\vecm_h^{(j)}(\vecx_n)
-k\vecv_h^{(j+1)}(\vecx_n)
\right|
=
\left|
\vecm_h^{(j)}(\vecx_n)+k\vecv_h^{(j+1)}(\vecx_n)
\right|
-1.
\end{equation}
It easily follows from $\left|\vecm_h^{(j)}(\vecx_n)\right|=1$ and 
$\vecv_h^{(j+1)}(\vecx_n)\cdot\vecm_h^{(j)}(\vecx_n)=0$ that 
\begin{equation*}
\left|
\vecm_h^{(j)}(\vecx_n)+k\vecv_h^{(j+1)}(\vecx_n)
\right|
\leq
\frac{1}{2}k^2\left|\vecv_h^{(j+1)}(\vecx_n)\right|^2 +1.
\end{equation*}
The above inequality and \eqref{E:4.8} yield
\[
\left|
\frac{\vecm_h^{(j+1)}(\vecx_n)-\vecm_h^{(j)}(\vecx_n)}{k}
-\vecv_h^{(j+1)}(\vecx_n)
\right|
\leq
\frac{1}{2}k\left|\vecv_h^{(j+1)}(\vecx_n)\right|^2.
\]
By using Lemma~\ref{lem:3.4} we deduce 
\begin{align*}
\left\|
\frac{\partial\vecm_{h,k}}{\partial t}(t)
-
\vecv_{h,k}(t)
\right\|_{\mL^1(D)}
\leq
ck
\left\|
\vecv_{h,k}(t)
\right\|^2_{D}
\quad\text{for }t\in [t_j,t_{j+1}).
\end{align*}
Integrating both sides of this inequality with respect 
to $t$ over an interval $[t_j,t_{j+1})$  and 
summing over $j$ from $0$ to $J-1$ yield, noting the
boundedness of $\{\norm{\vecv_{h,k}}{D_T}\}$,
\begin{equation*}
\left\|
\frac{\partial\vecm_{h,k}}{\partial t}
-
\vecv_{h,k}
\right\|_{\mL^1(D_T)}
\leq
ck
\left\|
\vecv_{h,k}
\right\|^2_{D_T}
\le
ck \goto 0 \quad\text{as } h,k \goto 0.
\end{equation*}
Thus $T_3\goto0$ as $h,k\goto0$ and $i\goto\infty$. It
follows from~\eqref{equ:mt} that
\[
|\inpro{\vecm_t-\vecv}{\vecpsi}_{D_T}| = 0
\quad\forall\vecpsi\in\mL^2(D_T).
\]
This proves \eqref{star:3}.

\bigskip
\noindent
\underline{Proof of~\eqref{star:4}:}

It is clear from the definition of $\vecm_{h,k}^-$ and $\vecm_{h,k}$ that for $t\in [t_j,t_{j+1})$ there holds
\begin{equation*}
\|
\vecm_{h,k}(t)-\vecm_{h,k}^-(t)
\|_D
=
\left\|
(t-t_j)
\frac{\vecm_h^{(j+1)}-\vecm_h^{(j)}}{k}
\right\|_D
\leq
k\left\|
\frac{\partial(\vecm_{h,k}(t,x))}{\partial t}
\right\|_D.
\end{equation*}
Integrating both sides of this inequality with respect to $t$ over an interval $[t_j,t_{j+1})$  and summing over $j$ from $0$ to $(J-1)$ yield
\begin{equation*}
\left\|
\vecm_{h,k}-\vecm_{h,k}^-
\right\|_{D_T}
\leq
k\left\|
\frac{\partial\vecm_{h,k}}{\partial t}
\right\|_{D_T}
\le
ck \goto 0 \text{ as } h,k\goto 0.
\end{equation*}
The above result and~\eqref{star:1}
imply~\eqref{star:4}.

\bigskip
\noindent
\underline{Proof of~\eqref{star:5}:}

Using Lemma~\ref{lem:3.5} and noting that $|\vecm_h^{(j)}(\vecx_n)|=1$ for $n=1,\cdots,N$, we deduce
\begin{equation*}
\left|
|\vecm_h^{(j)}(\vecx)|-1
\right|^2
\leq
Ch^2
\left|
\nabla\vecm_h^{(j)}(\vecx)
\right|^2
\quad\text{for all } \vecx\in D.
\end{equation*}
Integrating both sides of the above inequality on 
$[t_j,t_{j+1})\times D$, using~Lemma~\ref{lem:4.1}
and noting Remark~\ref{Rem:1}, we obtain
\begin{equation*}
\int_{t_j}^{t_{j+1}}\int_{D}\left|1-|\vecm_h^{(j)}(\vecx)|\right|^2\dvx\dt
\leq
ch^2
\int_{t_j}^{t_{j+1}}
\left\|
\nabla\vecm_h^{(j)}
\right\|^2_{D}
\leq
ckh^2.
\end{equation*}
Hence
\begin{equation*}
\int_{D_T}\left|1-|\vecm_{h,k}^-|\right|^2\dvx\dt
\goto 0 \text{ as } h,k\goto 0.
\end{equation*}
We infer from \eqref{star:4} that
\[
|\vecm|=1 \text{ a.e. in } D_T.
\]

\bigskip
\noindent
\underline{Proof of~\eqref{star:6},~\eqref{star:7} and~\eqref{star:8}:}

By using the same arguments as above, we obtain these
results, completing the proof of the lemma.
\end{proof}

We are now able to prove the main result of this paper.
\begin{theorem}\label{T:1}
Assume that $h$ and $k$ go to 0 with the following
conditions
\begin{equation}\label{equ:theta}
\begin{cases}
k = o(h^2) & \quad\text{when } 0 \le \theta < 1/2, \\
k = o(h) & \quad\text{when } \theta = 1/2, \\
\text{no condition} & \quad\text{when } 1/2<\theta\le1.
\end{cases}
\end{equation}
Then the limits
$\left(\vecm,\vecH\right)$ given by Lemma~\ref{lem:4.2}
is  a weak solution of 
the MLLG equations~\eqref{wE:LLG2}--\eqref{wE:Maxwell2}.
\end{theorem}
\begin{proof}
For any $\vecphi\in\C^{\infty}(D_T)$, 
$\veczeta \in\C^{\infty}(\wtd D_T)$, and 
$t\in [t_j,t_{j+1})$, we define
\[
\vecw_{h,k}(t,\cdot)
:=
I_{\mV_h}(\vecm_{h,k}^-\times\vecphi(t,\cdot))
\quad\text{and}\quad
\veczeta_h(t,\cdot):=I_{\mY_h}(\veczeta(t,\cdot)).
\]
In equations~\eqref{disE:LLG} and \eqref{disE:NewMax2},
replacing $\vecw_h^{(j)}$ and $\veczeta_h$ by 
$\vecw_{h,k}(t)$ and $\veczeta_h(t)$,
respectively,
and using Definition~\ref{Def:m}, 
we rewrite~\eqref{disE:LLG}--\eqref{disE:NewMax2} as
\begin{align*}
&-\lambda_1
\inpro{\vecm_{h,k}^{-}(t) \times\vecv_{h,k}(t)}{\vecw_{h,k}(t)}_{D}
+
\lambda_2
\inpro{\vecv_{h,k}(t)}{\vecw_{h,k}(t)}_{D}
\nn\\
&\quad=-\mu 
\inpro{\nabla (\vecm_{h,k}^{-}(t)+k\theta \vecv_{h,k}(t))}
{\nabla\vecw_{h,k}(t)}_{D} 
+
\mu
\inpro{\wtd\vecH_{h,k}(t)}{\vecw_{h,k}(t)}_{D},
\end{align*}
and
\[
\mu_0\inpro{\frac{\partial\vecH_{h,k}}{\partial t}(t)}{\veczeta_h(t)}_{\wtd D}
+\sig\inpro{\nabla\times\wtd\vecH_{h,k}(t)}{\nabla\times\veczeta_h(t)}_{\wtd D}
=-\mu_0\inpro{\vecv_{h,k}(t)}{\veczeta_h(t)}_{\wtd D}.
\]
Integrating both sides of these equations with respect 
to $t$ over an interval $[t_j,t_{j+1})$  and summing 
over $j$ from $0$ to $J-1$ yield
\begin{align} \label{InE:10}
&-\lambda_1
\inpro{\vecm_{h,k}^{-} \times\vecv_{h,k}}{\vecw_{h,k}}_{D_T}
+
\lambda_2
\inpro{\vecv_{h,k}}{\vecw_{h,k}}_{D_T}
\nn\\
&\quad=-\mu 
\inpro{\nabla (\vecm_{h,k}^{-}+k\theta \vecv_{h,k})}
{\nabla\vecw_{h,k}}_{D_T} 
+
\mu
\inpro{\wtd\vecH_{h,k}}{\vecw_{h,k}}_{D_T} 
\end{align}
and
\begin{equation}\label{equ:InE:11}
\mu_0\inpro{\frac{\partial\vecH_{h,k}}{\partial t}}{\veczeta_h}_{\wtd D_T}
+\sig\inpro{\nabla\times\wtd\vecH_{h,k}}{\nabla\times\veczeta_h}_{\wtd D_T}
=-\mu_0\inpro{\vecv_{h,k}}{\veczeta_h}_{\wtd D_T}.
\end{equation}

In order to prove that $\vecm$ and $\vecH$
satisfy~\eqref{wE:LLG2} and~\eqref{wE:Maxwell2},
respectively, we prove that
as $h$ and $k$ tend to $0$ there hold
\begin{align}
\inpro{\vecm_{h,k}^{-} \times\vecv_{h,k}}{\vecw_{h,k}}_{D_T}
&\goto
\inpro{\vecm \times\vecm_t}{\vecm\times\vecphi}_{D_T},\label{Conver:1}\\
\inpro{\vecv_{h,k}}{\vecw_{h,k}}_{D_T}
&\goto
\inpro{\vecm_t}{\vecm\times\vecphi}_{D_T},\label{Conver:2}\\
\inpro{\nabla \vecm_{h,k}^{-}}
{\nabla\vecw_{h,k}}_{D_T} 
&\goto
\inpro{\nabla \vecm}
{\nabla(\vecm\times\vecphi)}_{D_T}, \label{Conver:3}\\
k\inpro{\nabla \vecv_{h,k}}
{\nabla\vecw_{h,k}}_{D_T} 
&\goto
0, \label{Conver:4}\\
\inpro{\wtd\vecH_{h,k}}{\vecw_{h,k}}_{D_T}
&\goto
\inpro{\vecH}{\vecm\times\vecphi}_{D_T},\label{Conver:5}
\end{align}
and
\begin{align}
\inpro{\frac{\partial\vecH_{h,k}}{\partial t}}{\veczeta_h}_{\wtd D_T}
&\goto
\inpro{\vecH_t}{\veczeta}_{\wtd D_T},\label{Conver:6}\\
\inpro{\nabla\times\wtd\vecH_{h,k}}{\nabla\times\veczeta_h}_{\wtd D_T}
&\goto
\inpro{\nabla\times\vecH}{\nabla\times\veczeta}_{\wtd D_T},\label{Conver:7}\\
\inpro{\vecv_{h,k}}{\veczeta_h}_{\wtd D_T}
&\goto
\inpro{\vecm_t}{\veczeta}_{\wtd D_T}.\label{Conver:8}
\end{align}
We now prove~\eqref{Conver:1} and~\eqref{Conver:4}; 
the others can be obtained in the same manner.

Using the triangular inequality and Holder's inequality, we estimate
\[I_{h,k}
:=
\left|\inpro{\vecm_{h,k}^{-} \times\vecv_{h,k}}{\vecw_{h,k}}_{D_T}
-
\inpro{\vecm \times\vecm_t}{\vecm\times\vecphi}_{D_T}\right|
\]
as follows:
\begin{align*}
 I_{h,k}
&\leq
\left|\inpro{\vecm_{h,k}^{-} 
\times\vecv_{h,k}}{\vecw_{h,k}-\vecm_{h,k}^{-}\times\vecphi}_{D_T}\right|
+
\left|\inpro{\vecm_{h,k}^{-} 
\times\vecv_{h,k}}{(\vecm_{h,k}^{-}-\vecm)\times\vecphi}_{D_T}\right|\\
&\quad+
\left|\inpro{(\vecm_{h,k}^{-}-\vecm) 
\times\vecv_{h,k}}{\vecm\times\vecphi}_{D_T}\right|
+
\left|\inpro{\vecm 
\times(\vecv_{h,k}-\vecm_t)}{\vecm\times\vecphi}_{D_T}\right|\\
&\leq
\|\vecm_{h,k}^{-}\|_{\mL^{\infty}(D_T)}
\|\vecv_{h,k}\|_{D_T}
\|\vecw_{h,k}-\vecm_{h,k}^{-}\times\vecphi\|_{D_T}\\
&\quad+
\|\vecm_{h,k}^{-}\|_{\mL^{\infty}(D_T)}
\|\vecv_{h,k}\|_{D_T}
\|\vecm-\vecm_{h,k}^{-}\|_{D_T}
\|\vecphi\|_{\mL^{\infty}(D_T)}\\
&\quad+
\|\vecm-\vecm_{h,k}^{-}\|_{D_T}
\|\vecv_{h,k}\|_{D_T}
\|\vecphi\|_{\mL^{\infty}(D_T)}\\
&\quad+
\|\vecv_{h,k}-\vecm_t\|_{D_T}
\|\vecphi\|_{\mL^{\infty}(D_T)} \\
&\le
c\left(
\|\vecw_{h,k}-\vecm_{h,k}^{-}\times\vecphi\|_{D_T}
+
\|\vecm-\vecm_{h,k}^{-}\|_{D_T}
+
\|\vecv_{h,k}-\vecm_t\|_{D_T}
\right),
\end{align*}
where we have used \eqref{bound:minf} and
Lemma~\ref{lem:4.1}, noting Remark~\ref{Rem:1}.
The interpolation operators $I_{\mV_h}$ and $I_{\mY_h}$
have the following properties (see e.g., \cite{Braess07} 
and~\cite{Monk03})
\begin{equation}\label{equ:int ope}
\begin{aligned}
&\|\vecm_{h,k}^-\times\vecphi-\vecw_{h,k}\|_{\mL^2([0,T],\mH^1(D))}
\le
Ch
\|\vecm_{h,k}^-\|_{\mH^1(D_T)}\|\vecphi\|_{\mW^{2,\infty}(D_T)},
\\
&\|\veczeta(t)-\veczeta_h(t)\|_{\wtd D}
+
\|\nabla\times(\veczeta(t)-\veczeta_h(t))\|_{\wtd D}
\le
Ch
\|\nabla^2\veczeta\|_{\wtd D}.
\end{aligned}
\end{equation}
This implies
\[
\lim_{k,h\goto 0} I_{h,k} = 0,
\]
proving~\eqref{Conver:1}.

In order to prove~\eqref{Conver:4} we first note that
\begin{align*}
\|\nabla\vecw_{h,k}\|_{D_T}
&\leq
\|\nabla(\vecm_{h,k}^-\times\vecphi-\vecw_{h,k})\|_{D_T}
+
\|\nabla(\vecm_{h,k}^-\times\vecphi)\|_{D_T} \\
&\leq
ch
\|\vecm_{h,k}^-\|_{\mH^1(D_T)}
\|\vecphi\|_{\mW^{2,\infty}(D_T)}
+
\|\nabla\vecm_{h,k}^-\|_{D_T}
\|\nabla\vecphi\|_{\mL^{\infty}(D_T)} \\
&\leq
c\|\vecphi\|_{\mW^{2,\infty}(D_T)},
\end{align*}
where we have used~\eqref{equ:int ope} and the
boundedness of $\norm{\vecm_{h,k}^-}{\mH^1(D_T)}$.
Now using Holder's inequality we obtain
\begin{equation}\label{Conver:4.1}
k\inpro{\nabla \vecv_{h,k}}
{\nabla\vecw_{h,k}}_{D_T} 
\leq
ck
\|\nabla\vecv_{h,k}\|_{D_T}.
\end{equation}
It is straightforward from~\eqref{E:4.7} that
$\{\|\nabla\vecv_{h,k}\|_{D_T}\}$
is bounded when $\theta\in(\frac{1}{2},1]$. 
Hence, taking the limit as $k$ and $h$ tend to $0$ 
in~\eqref{Conver:4.1} yields~\eqref{Conver:4} for these
values of~$\theta$.

When $\theta\in[0,\frac{1}{2}]$, using the inverse estimate we obtain
\begin{equation*}
\sum_{i=0}^{j-1} \|\nabla\vecv^{(i+1)}_h\|^2_{D}
\leq
ch^{-2}\sum_{i=0}^{j-1} \|\vecv^{(i+1)}_h\|^2_{D},
\end{equation*}
or equivalently,
\begin{equation*}
\|\nabla\vecv_{h,k}\|_{D_T}
\leq
ch^{-1}\|\vecv_{h,k}\|_{D_T}.
\end{equation*}
Hence under the assumption~\eqref{equ:theta}, the
inequality \eqref{Conver:4.1} becomes
\[
k\inpro{\nabla \vecv_{h,k}}
{\nabla\vecw_{h,k}}_{D_T} 
\leq
ckh^{-1}\|\vecv_{h,k}\|_{D_T}
\le
ckh^{-1}
\]
when $\theta\in[0,1/2]$. Therefore, under the
assumption~\eqref{equ:theta} there holds
\[
k\inpro{\nabla \vecv_{h,k}}
{\nabla\vecw_{h,k}}_{D_T} 
\goto
0. 
\]

We now prove~\eqref{equ:m0}.
Since $\vecm_h^0=I_{\mV_h}(\vecm_0)$, the sequence $\{\vecm_h^0\}$ converges to $\vecm_0$ in $\mL^2(D)$ as $h$ tends to $0$. Using the weak continuity of the trace operator we obtain that $\vecm(0,\cdot)=\vecm_0$ in the sense of traces.

Finally, applying weak lower semicontinuity of norms in inequality \eqref{E:4.1} we obtain the energy inequality~\eqref{InE:Energy}, which completes the proof.
\end{proof}

\section{Numerical experiments }\label{Sec:Simu}
In order to carry out physically relevant experiments,
the initial fields $\vecm_0$, $\vecH_0$ must satisfy 
condition~\eqref{E:Cond2}. This can be achieved by taking
\begin{equation*}
\vecH_0 =\vecH_0^* -  \chi_D\vecm_0,
\end{equation*}
where $\dive \vecH_0^* = 0$ in $\wtd D$. 
In our experiment, for simplicity, we choose
$\vecH_0^*$ to be a constant.
We solve an academic example with $D=\wtd D=(0,1)^3$ and 
\begin{align*}
\vecm_0(\vecx)
&=
\begin{cases}
(0,0,-1), &\quad |\vecx^*|\geq \frac{1}{2},\\
(2\vecx^*A,A^2-|\vecx^*|^2)/(A^2+|\vecx^*|^2),
&\quad |\vecx^*|\leq \frac{1}{2},
\end{cases}\\
\vecH_0^*(\vecx)
&=
(0,0,H_s),\quad\vecx\in \wtd D,
\end{align*}
where $\vecx=(x_1,x_2,x_3)$, $\vecx^*=(x_1-0.5,x_2-0.5,0)$ and $A=(1-2|\vecx^*|)^4/4$.
The constant $H_s$ represents the strength of $\vecH_0$
in the  $x_3$-direction. We compute the experiments for
$H_s=0$, $\pm 30$, $\pm 100$ and $\pm 1000$.
We set the values for the other parameters
in~\eqref{E:LL} and~\eqref{E:Maxwell2} as
$\lam_1=\lam_2=\mu_0=\sig=1$.

The domain $D$ is partitioned into uniform cubes with
the mesh size $h=1/2^3$, where each cube consists of 
six tetrahedra. We choose the time step $k=10^{-3}$ and
the parameter $\theta$ in Algorithm~\ref{Algo:1} to be
$0.7$.
The construction of the basis functions for $\mW_h^{(j)}$ 
and $\mY_h$ in this algorithm is discussed
in~\cite{LeTran_2012}. At each iteration we need to
solve a linear system of size $(2N+M)\times(2N+M)$,
recalling that $N$ is the number of vertices and $M$ is
the number of edges in the triangulation. The code is
written in Fortran90.

The evolution of 
$\|\nabla\vecm_{h,k}\|_D$, $\|\vecH_{h,k}\|_{\wtd D}$ 
and $\|\nabla\times\vecH_{h,k}\|_{\wtd D}$  are 
depicted in Figures~\ref{Fig:Eex}, \ref{Fig:Eh} 
and~\ref{Fig:Ee}, respectively.
Figure~\ref{fig:Etotal} shows that the solution
satisfies condition~\eqref{InE:Energy} in
Definition~\ref{Def:WeakSo1}.
\begin{figure}
\centering
\includegraphics[width=15cm,height=9cm]{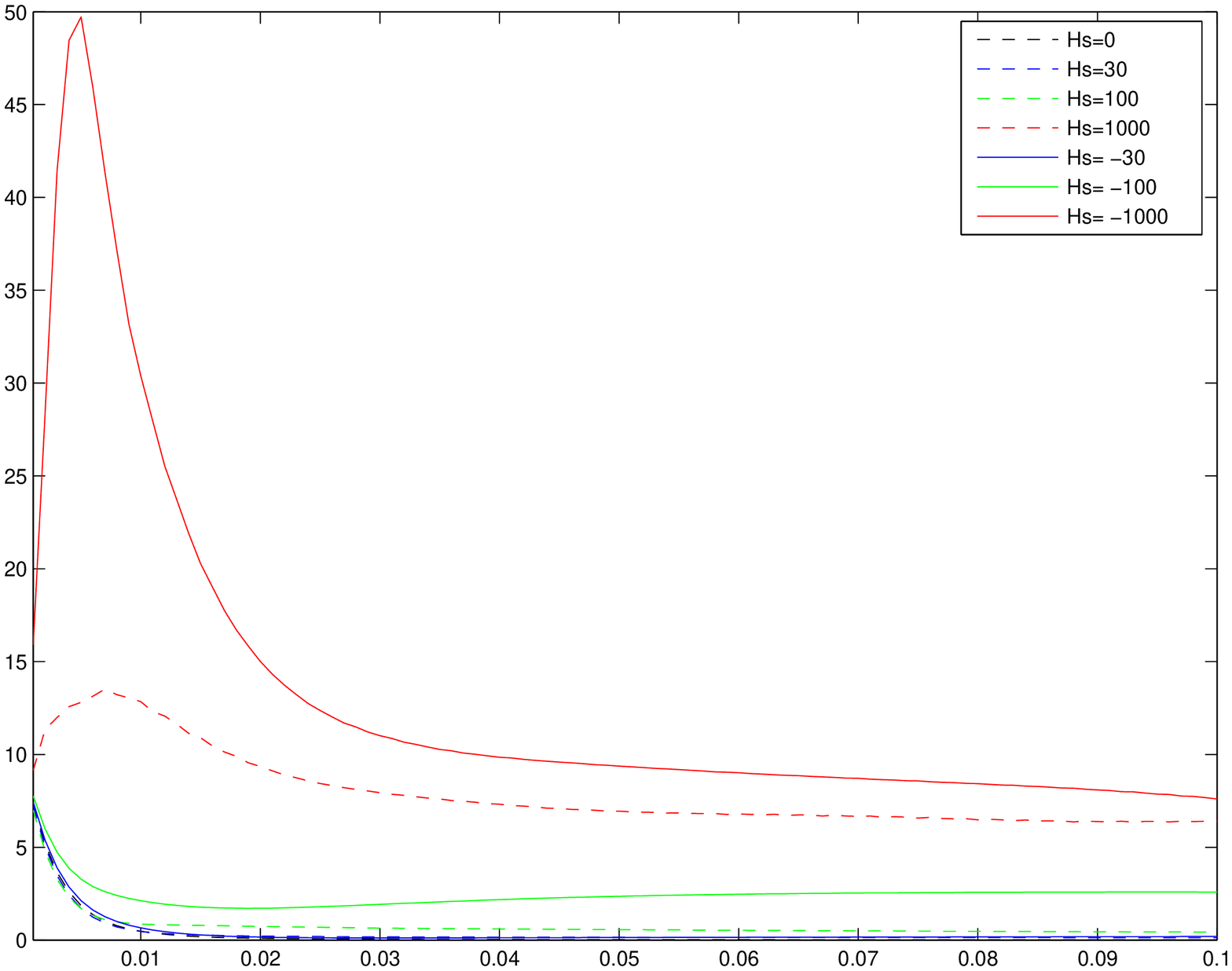}
\caption{Plot of $t\mapsto\|\nabla\vecm_{h,k}(t)\|_D$ .}\label{Fig:Eex}
\end{figure}
\begin{figure}
\centering
\includegraphics[width=15cm,height=9cm]{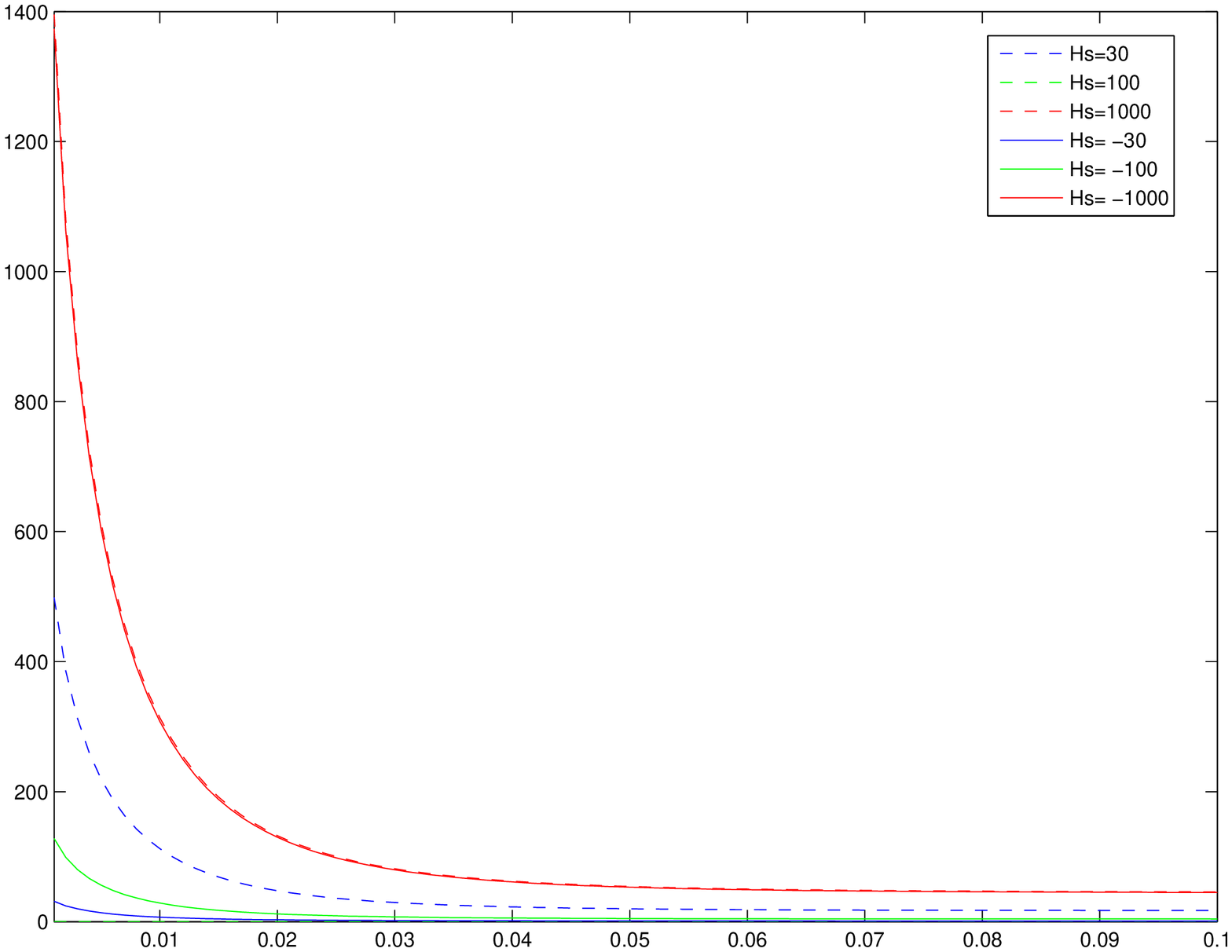}
\caption{Plot of $t\mapsto\|\vecH_{h,k}(t)\|_{\wtd D}$ .}\label{Fig:Eh}
\end{figure}
\begin{figure}
\centering
\includegraphics[width=15cm,height=9cm]{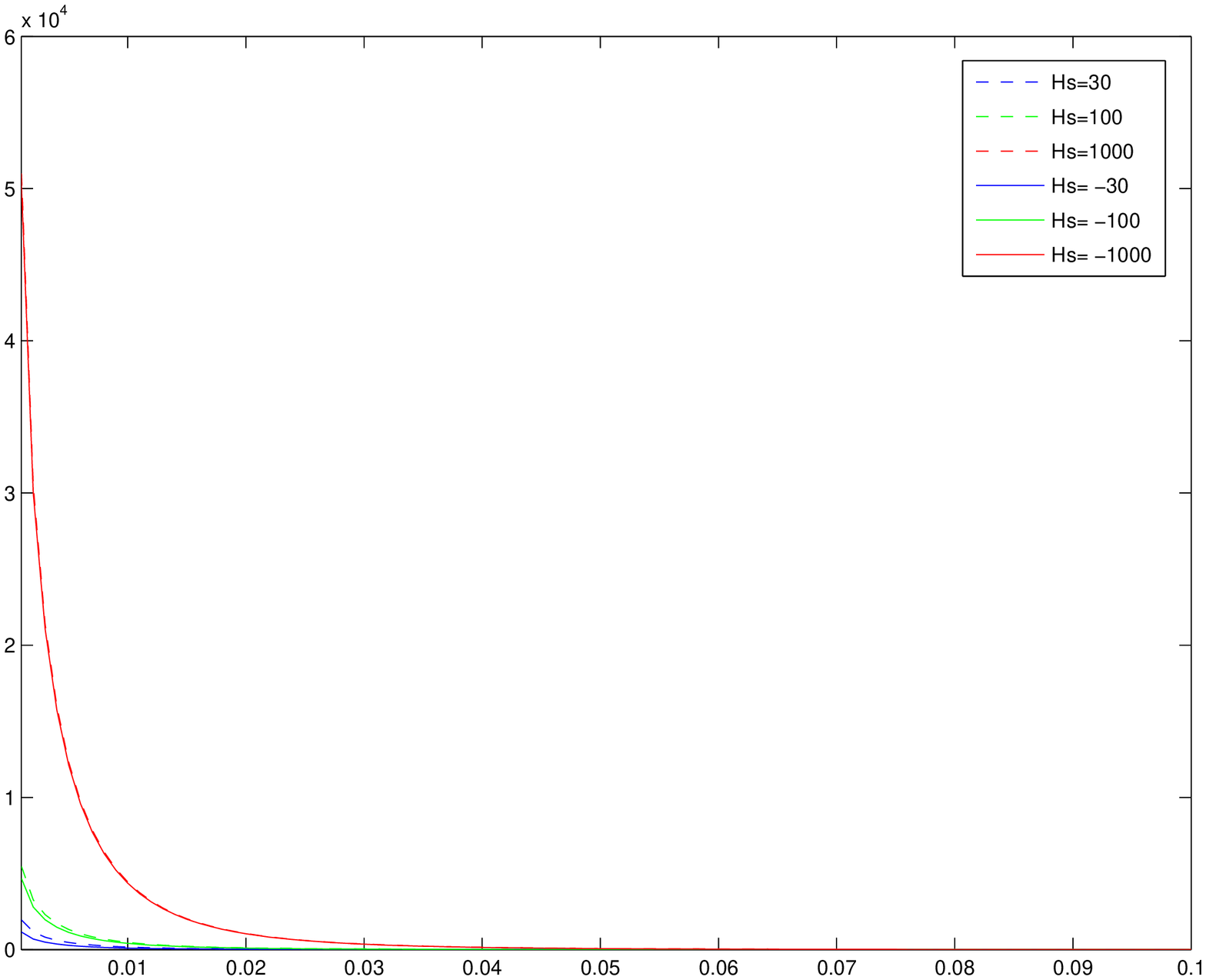}
\caption{Plot of $t\mapsto\|\nabla\times\vecH_{h,k}(t)\|_{\wtd D}$ .}\label{Fig:Ee}
\end{figure}
\begin{figure}
\centering
\includegraphics[width=15cm,height=9cm]{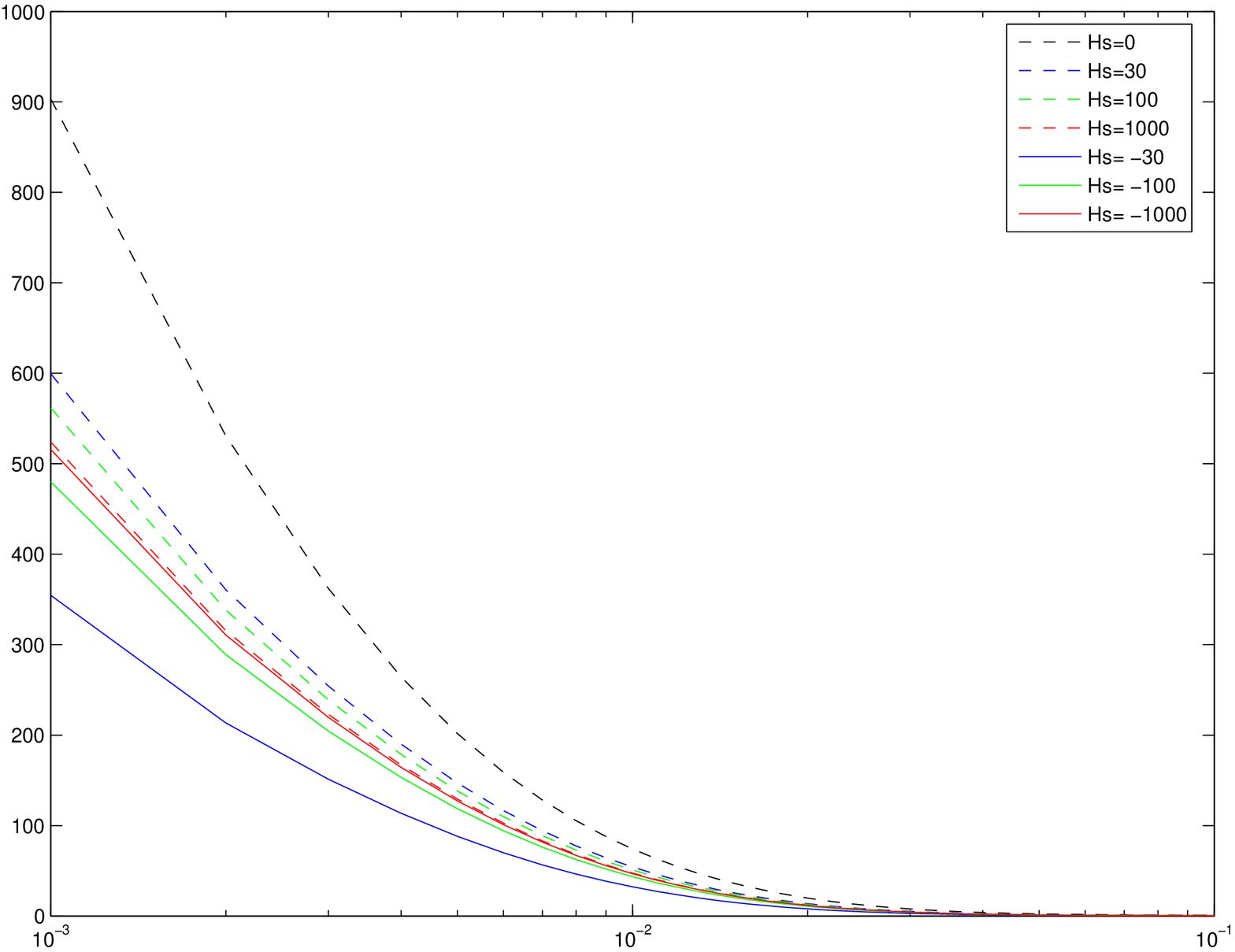}
\caption{Plot of $\log t\mapsto\cE(t)$}\label{fig:Etotal}
\end{figure}
\begin{remark}
By the time this paper was written up, we learnt that Ba\v
nas, Page and Praetorius \cite{BanPagPra12}
independently solved a similar problem. They also used
a linear scheme similar to our scheme, even
though their variational formulation was different.
\end{remark}

\def\cprime{$'$}

\end{document}